\documentclass[graybox,vecarrow]{svmult}

\usepackage{type1cm}        
\usepackage{makeidx}         
\usepackage{graphicx}        
\usepackage{multicol}        
\usepackage[bottom]{footmisc}

\usepackage{newtxtext}       %
\usepackage[varvw]{newtxmath}       

\makeindex             

\newcommand{\bsa}{{\boldsymbol{a}}}

\newcommand{\bsc}{{\boldsymbol{c}}}

\newcommand{\bsk}{{\boldsymbol{k}}}

\newcommand{\bsx}{{\boldsymbol{x}}}
\newcommand{\bsy}{{\boldsymbol{y}}}
\newcommand{\bsz}{{\boldsymbol{z}}}

\newcommand{\bszero}{{\boldsymbol{0}}} 
\newcommand{\bsone}{{\boldsymbol{1}}}  

\newcommand{\rd}{{\mathrm{d}}}

\newcommand{\bbN}{{\mathbb{N}}}

\newcommand{\bbR}{{\mathbb{R}}}

\newcommand{\bbZ}{{\mathbb{Z}}}
\newcommand{\N}{{\mathbb{N}}} 
\newcommand{\Z}{{\mathbb{Z}}} 

\DeclareSymbolFont{bbold}{U}{bbold}{m}{n}
\DeclareSymbolFontAlphabet{\mathbbold}{bbold}

\usepackage[space]{grffile} 

\begin{document}

\title*{Where are the logs?}
\author{Art B. Owen and Zexin Pan}
\institute{Art B. Owen \at Stanford University, \email{owen@stanford.edu}
\and Zexin Pan \at Stanford University \email{zep002@stanford.edu}}
%
%
\maketitle

\abstract{
The commonly quoted
error rates for  QMC integration with an infinite low discrepancy
sequence is $O(n^{-1}\log(n)^r)$ with $r=d$
for extensible sequences
and $r=d-1$ otherwise.
Such rates hold uniformly over all $d$ dimensional integrands
of Hardy-Krause variation one when using $n$
evaluation points.  Implicit in those
bounds is that for any sequence of QMC points,
the integrand can be chosen to depend on $n$.
In this paper we show that
rates with any $r<(d-1)/2$ can hold when $f$ is held fixed as $n\to\infty$.
This is accomplished following a suggestion of Erich
Novak to use some unpublished results of Trojan
from the 1980s as given in the information
based complexity monograph of Traub, Wasilkowski
and Wo\'zniakowski.
The proof is made by applying a technique of Roth
with the theorem of Trojan. The proof is non constructive
and we do not know of any integrand of bounded
variation in the sense of Hardy and Krause
for which the QMC error exceeds $(\log n)^{1+\epsilon}/n$
for infinitely many $n$ when using a digital
sequence such as one of Sobol's.
An empirical search when $d=2$ for integrands
designed to exploit known weaknesses in certain
point sets showed no evidence that $r>1$ is
needed. An example with $d=3$ and $n$
up to $2^{100}$ might possibly require $r>1$.
}

\section{Introduction}




In this article, we study the asymptotic
error rates for integration by quasi-Monte Carlo (QMC) as $n\to\infty$
while $f$ is fixed.
Most of the error upper bounds in QMC are based on fooling functions
$f_n$ that, given $n$ integration points, are poorly integrated.
By contrast, most of the
published empirical results follow the integration
error for a single integrand $f$ as $n$ increases.
The upper bounds have us play against an adaptive adversary choosing
an unfavorable $f_n$ at each sample size $n$
instead of keeping $f$  fixed as $n\to\infty$.
The error bounds, that we describe in more detail below,
are typically $O(\log(n)^r/n)$ where $r$ can be as large as the
dimension of the integrand's domain.
These bounds can be enormous and, to our knowledge, there has never
been an integrand exhibited where a standard QMC point set
is shown to need $r>1$.
That raises the question of whether $r>1$ is simply
a consequence of the adversarial formulation.
The alternative is that some function $f$ is a `persistent fooling function'
causing large errors for infinitely many $n$.
In an earlier version of this article we posed a question about whether
{\sl any} integrand of bounded variation in the sense of Hardy and Krause (BVHK)
on $[0,1]^d$ for any $d\geqslant1$
has an integration error above $c\log(n)^r/n$ for infinitely many $n$
with $r>1$ and $c>0$ using a digital sequence such as Sobol's for the integration points.
For background on bounded variation in the sense of Hardy and Krause
or in the sense of Vitali, we refer the reader to \cite{variation}.

We owe a great debt to Erich Novak who pointed us to some unpublished
work of Trojan described in detail in Chapter 10 of
the information based complexity monograph of Traub, Wasilkowski
and Wo\'zniakowsk \cite{trau:wasi:wozn:1988}.
Trojan's work is about very general problems of computing
linear operators on Banach spaces based on the values of $n\to\infty$ linear functionals.
He shows that the adversarial worst case convergence rate is also very nearly the attained
rate for some specific problem instances.
In the QMC context, that work pertains to a single integrand $f$
as the number $n$ of evaluation points diverges to infinity.
A consequence of that work is that for any infinite
sequence of integration points, there are indeed integrands
in BVHK$[0,1]^d$ with an absolute error larger than $c\log(n)^{(d-1)/2}/(n\log\log(n))$
infinitely often, for any $c>0$.
Furthermore, those integrands are present within a reproducing kernel Hilbert space (RKHS)
on a certain unanchored space.  They are dense in that space, though
this does not mean that the usual Gaussian processes on such spaces
give them positive measure.
We only get $r=(d-1)/2$ logarithmic factors instead of $d$ or $d-1$
of them.  The explanation is that we use an $L^2$ bound
just like Roth \cite{roth:1954} used in getting a lower bound on star discrepancy.
A different analysis might yield larger $r$.
The $\log\log(n)$ factor in the denominator
can be replaced by a sequence that diverges more slowly.

We have not been able to construct a function in BVHK$[0,1]^d$
that provably needs $r>1$ powers of $\log(n)$
for a Sobol' sequence \cite{sobo:1967:tran}
or the Halton \cite{halt:1960}
sequence, even when exploiting known weaknesses of commonly used QMC sequences.
So, we are left to wonder: where are the logs?

An outline of this paper is as follows.
Section~\ref{sec:back} presents some results from the QMC literature
and introduces notation on some QMC sequences.
Section~\ref{sec:proofofbound} proves our main result
described above on existence of persistent fooling functions.
Section~\ref{sec:d=1} looks at the case $d=1$
to exhibit some example functions requiring $r=1$
for the van der Corput sequence: $f(x)=1\{x<2/3\}$
and $f(x)=x$.
Section~\ref{sec:d=2} computes error for some $d=2$
dimensional problems. The Halton and Sobol'
points there are closely related to van der Corput
points, yet two dimensional generalizations of the
problematic integrands from Section~\ref{sec:d=1}
fail to show a need for $r>1$.
In fact some of the empirical results are more consistent
with an $O(1/n)$ error.
Section~\ref{sec:bigm} computes a local discrepancy
$\delta(\bsz$)
for  Sobol' nets with $d=2,3$ and $1\leqslant m\leqslant100$ where all components of $\bsz$ equal $2/3$
chosen because $2/3$ is difficult to approximate by dyadic rationals.
It also includes $d=4$ for $1\leqslant m\leqslant 50$.
The cases with $d>2$ are the closest we have found to needing $r>1$
but are inconclusive.
Section~\ref{sec:discussion} discusses these results.

\section{Background}\label{sec:back}

From the Koksma-Hlawka inequality \cite{hick:2014}
combined with convergence rates for the star discrepancy \cite{nied:1992},
we get the widely quoted convergence rates
for the error in quasi-Monte Carlo integration
of a function $f:[0,1]^d\to\bbR$.
An integrand $f$ of bounded variation in the
sense of Hardy and Krause, written $f\in\mathrm{BVHK}[0,1]^d$,
can be integrated with error $O(n^{-1}(\log n)^{d-1})$
using $n$ function evaluations.  If we must use the first $n$ points
of an infinite sequence, then the rate
$O(n^{-1}(\log n)^{d})$ is attainable.
This article is mostly about the infinite sequence version.
Both of these rates are often written $O(n^{-1+\epsilon})$
where $\epsilon$ can be any positive constant
but $\log(n)^d\gg n^\epsilon$ for many use cases of interest.

For high dimensional problems, such powers of
$\log(n)$ are enormous and then
there is genuine uncertainty about whether
$O(n^{-1}(\log n)^{d})$ is better than the
root mean squared error (RMSE) of $O(n^{-1/2})$ from plain
Monte Carlo (MC) at practically relevant  $n$.
These rates omit three implied constants: one in
the star discrepancy (see \cite{faur:lemi:2014} for information),
one in the total variation of $f$
and the third one is the standard deviation of $f$.
These unknown constants contribute to
uncertainty about the $n$ at which QMC would
outperform MC.
A further complication is that the Koksma-Hlawka
bound is for a worst case integrand.
The situation is quite different in Monte Carlo
(MC) where the rate $\sigma n^{-1/2}$ holds for all finite $n$ making
it simultaneously a guide to how accuracy progresses
for a single integrand of variance $\sigma^2$ and
the RMSE formula (upper and lower bound)
for all integrands of variance~$\sigma^2$.

That observed errors for realistic $n$ and large $d$
do not follow a trend like $\log(n)^d/n$
was reported by Schlier \cite{schl:2004} among others.
That work also found that the variance of $f(\bsx)$
was more useful than its total variation in explaining the empirical accuracy
of QMC integration on test functions, despite the fact
that proved theoretical bounds for QMC error use total
variation and variance does not require any of the smoothness that
QMC relies on.
Many papers include empirically estimated convergence
rates for individual $f$ found by fitting a regression model
for log error versus $\log(n)$. See for instance L'Ecuyer \cite{lecu:2018}.
We do not see results that look like a large power of $\log(n)$
is present.

This mismatch between empirical results and theoretical ones
is troubling.  Empirical results alone don't give enough confidence
that they will apply to future problems.  Similarly, bounds that are
favorable (but asymptotic) or unfavorable (but worst case) could
also fail to provide a reliable guide to attained accuracy.
This mismatch has brought practical difficulties.
For instance, the logarithmic powers
in the Koksma-Hlawka bound led Bratley, Fox and Niederreiter \cite{brat:fox:nied:1992}
to limit their software to $d\leqslant 12$.

For some randomizations of digital nets
the RMSE is $O(n^{-1/2})$ whenever $f\in L^2[0,1]^d$ \cite{snxs}
and is also $O(\log(n)^{(d-1)/2}/n)$ under further smoothness
conditions \cite{smoovar,localanti,yue:mao:1999}.
In such cases the large powers of $\log(n)$ are
subject to a simultaneous $O(n^{-1/2})$ bound that
limits how much worse randomized QMC can be compared to MC
for finite $n$.
It would be interesting to know whether something like that also
holds for plain QMC. Perhaps the coefficient of $\log(n)^r/n$ is ordinarily
very small, or the effect is only relevant for impractically large $n$ or
perhaps not even present for most commonly investigated integrands.
For a survey of randomized QMC see L'Ecuyer and Lemieux \cite{lecu:lemi:2002}.

We conclude this section by describing $(t,m,d)$-nets and $(t,d)$-sequences
using the formulation from Niederreiter \cite{nied:1987}.
Let $b\geqslant2$ be an integer.
For $\bsk = (k_1,\dots,k_d)\in\N_0^d$ and
$\bsc = (c_1,\dots,c_d)\in\Z_0^d$
with $0\leqslant c_j<b^{k_j}$ the half open hyper-rectangle
\begin{align}\label{eq:elemint}
E(\bsk,\bsc) = \prod_{j=1}^d\Bigl[\frac{c_j}{b^{k_j}}, \frac{c_j+1}{b^{k_j}}\Bigr)
\end{align}
is called an elementary interval in base $b$.
It has volume $b^{-|\bsk|}$ where $|\bsk|=\sum_{j=1}^dk_j$.
We define its indicator function as
$$
I_{\bsk,\bsc}(\bsx)
=I_{\bsk,\bsc}(\bsx;b)
= 1\{\bsx\in E(\bsk,\bsc)\}.
$$

For integers $m\geqslant t\geqslant0$, $b\geqslant2$, $n=b^m$ and $d\geqslant1$
the points $\bsx_0,\dots,\bsx_{n-1}\in[0,1]^d$ are
a $(t,m,d)$-net in base $b$ if
$$
\frac1n\sum_{i=0}^{n-1} I_{\bsk,\bsc}(\bsx_i) =
\int_{[0,1]^d}I_{\bsk,\bsc}(\bsx)\rd\bsx=b^{-|\bsk|}
$$
holds for all elementary intervals with $|\bsk|\leqslant m-t$.
Other things being equal, smaller $t$ are better and
$t=0$ is best, but the choices of $d$, $b$, and $m$
impose a lower bound on the possible $t$ which
may rule out $t=0$.
The minT project of \cite{schu:schm:2009}
tracks the best known values of $t$ as well as some
lower bounds on $t$.

In this paper we emphasize infinite sequences.
The sequence $\bsx_i\in[0,1]^d$ for integers $i\geqslant0$
is a $(t,d)$-sequence in base $b$ if
$\bsx_{rb^m},\dots,\bsx_{(r+1)b^m-1}$ is a $(t,m,d)$-net
in base $b$ for all $m\geqslant t$ and all integers $r\geqslant0$.
These are extensible $(t,m,d)$-nets in that
the first $b^\ell$ points of a $(t,d)$-sequence form a $(t,m+\ell,d)$-net
for any integer $\ell\geqslant1$.
The most used $(t,d)$-sequences are the
$(0,d)$-sequences in prime bases $p\geqslant d$
of Faure \cite{faur:1982} and the $(t,d)$-sequences
in base $2$ of Sobol' \cite{sobo:1967:tran}.

We will make special use of the
van der Corput sequences in base $b\geqslant2$.
These are $(0,1)$-sequences in base $b$.
If we write the natural number $i=\sum_{k=1}^\infty i_k b^{k-1}$
with digits $i_k\in\{0,1,\dots,b-1\}$, then the sum only has $K(i)<\infty$
nonzero terms and we then set
$x_i = \sum_{k=1}^{K(i)}b^{-k}i_k$.
This sequence has star discrepancy $D^*_n=O(\log(n)/n)$.
For $n=b^m$ it is a left endpoint rule containing points
$i/n$ for $0\leqslant i< n$ and so it has $D^*_n = 1/n$
by \cite[Theorem 2.6]{nied:1992}.
The original van der Corput sequence in base $b=2$
is from \cite{vand:1935:I}.


\section{Proof of the lower bound}\label{sec:proofofbound}

We begin with a general theorem on worst-case errors.
Later we specialize it to the QMC setting.

\begin{theorem}\label{thm:asymptoticrate}
Let $(F,\Vert\cdot\Vert)$ be a Banach space and $S$ be a linear functional on $F$. For a sequence of continuous linear functionals $L_n$ on $F$, define
\begin{align*}
N_n(f)&=(L_1(f),\ldots,L_n(f)),\quad\text{and}\\
r_n&=\sup\bigl\{|S(f)|\bigm| f\in F, N_n(f)=\bszero, \Vert f\Vert\leqslant 1\bigr\}.
\end{align*}
Then for any sequence of mappings $\phi_n$ from $\mathbb{R}^n$ to $\mathbb{R}$, there exists $f\in F$ such that
\begin{equation}\label{eqn:target}
    \limsup_{n\to \infty}\frac{ |S(f)-\phi_n(N_n(f))|}{(\log\log n)^{-1}r_n}=+\infty.
\end{equation}
\end{theorem}
\begin{proof} This follows from Theorem 2.1.1 in Chapter 10 of
\cite{trau:wasi:wozn:1988}
who cite unpublished work by Trojan.
\end{proof}

We will set $F$ to be some reproducing kernel Hilbert space (RKHS) contained in BVHK$[0,1]^d$, $S(f)$ to be $\int_{[0,1]^d}f(\bsy)\,\rd\bsy$ and $N_n(f)$ to be $(f(\bsx_0),\ldots,f(\bsx_{n-1}))$. We note that evaluation at $\bsx_i$ is a continuous linear functional in an RKHS. As the theorem suggests, we can use $r_n/\log\log n$ as a lower bound on the asymptotic convergence rate achievable by all functions in BVHK$[0,1]^d$, so it remains to determine a lower bound on $r_n$.

To derive such a lower bound, we will apply the proof techniques used in Roth's lower bound on the $L^2$ discrepancy.
See Chen and Travaglini \cite{chen:trav:2009} for a nice summary.
Dick, Hinrichs and Pillichshammer \cite{dick:hinr:pill:2015}
use this strategy to prove that the worst-case error of any equal-weight quadrature rule is
lower bounded by $\Omega(n^{-1}(\log n)^{(d-1)/2})$ when $F$ is the RKHS with kernel $K(\bsx,\bsy)=\prod_{j=1}^d (1+\min(x_j,y_j))$. Wozniakowski \cite{wozn:1991} points out that the same strategy still works if the equal weight requirement is removed.
Below we illustrate Roth's technique by showing that $r_n=\Omega(n^{-1}(\log n)^{(d-1)/2})$ if $F$ is chosen to be the  RKHS with kernel
 \begin{align}\label{eqn:kernel}
     K(\bsx,\bsy)=\prod_{j=1}^d \frac{4}{3}+\frac{1}{2}\Bigl(x_j^2+y_j^2-x_j-y_j-|x_j-y_j|\Bigr).
 \end{align}
This is the unanchored space introduced in
\cite{dick:sloa:wang:wozn:2004}.
It has the inner product
\begin{align}\label{eqn:innerproduct}
    (f,g)=\sum_{u\subseteq\{1,\ldots,d\}}\int_{[0,1]^{|u|}}\bigg(\int_{[0,1]^{d-|u|}}\frac{\partial^{|u|}f}{\partial \bsy_u}(\bsy)\,\rd\bsy_{-u}\bigg) \bigg(\int_{[0,1]^{d-|u|}}\frac{\partial^{|u|}g}{\partial \bsy_u}(\bsy)\,\rd\bsy_{-u}\bigg)\, \rd\bsy_u
\end{align}
where ${\partial^{|u|}f}/{\partial \bsy_u}$ is the partial derivative of $f$ taken once with respect to
each $y_j$ with $j\in u$. 
Any $f$ belonging to this RKHS has mixed partial derivative
${\partial^{|u|}f}/{\partial \bsy_u}\in L^2$ for any $u\subseteq\{1,\ldots,d\}$, so $f$ belongs to BVHK$[0,1]^d$
using equation (5) and Proposition~13 of \cite{variation}.

Letting $\bsone\in F$ be the function equal to $1$ for all $\bsy\in[0,1]^d$,
it is straightforward to verify that
$$(f,\bsone)=\int_{[0,1]^{d}}f(\bsy) \,\rd\bsy=S(f).$$
In other words, the function $\bsone$ is the Riesz representation of integration over $[0,1]^d$. Moreover, by the reproducing property
$(f,K(\bsx_i,\cdot))=f(\bsx_i).$
Therefore
\begin{align*}
    r_n&=\sup\bigl\{|S(f)|\mid f\in F, f(\bsx_0)=\cdots=f(\bsx_{n-1})=0, \Vert f\Vert\leqslant 1\bigr\}\\
    &=\sup\bigl\{|(f,\bsone)|\mid f\in F, (f,K(\bsx_0,\cdot))=\cdots=(f,K(\bsx_{n-1},\cdot))=0, \Vert f\Vert\leqslant 1\bigr\}.
\end{align*}
This is the well-known least squares projection problem. The maximizer is proportional to the projection of $\bsone$ into the orthogonal complement of the linear span of $\{K(\bsx_0,\cdot),\ldots,K(\bsx_{n-1},\cdot)\}$.
Therefore
\begin{align*}
    r_n=\min_{a_0,\ldots,a_{n-1}} \Bigl\Vert\bsone-\sum_{i=0}^{n-1} a_i K(\bsx_i,\cdot)\Bigr\Vert.
\end{align*}
Now we prove that $\Vert\bsone-\sum_{i=0}^{n-1} a_i K(\bsx_i,\cdot)\Vert=\Omega(n^{-1}(\log n)^{(d-1)/2})$
for any choice of $a_0,\dots,a_{n-1}$, including  $a_i=1/n$ as used in QMC.
\begin{theorem}\label{thm:Roth}
Let $K$ be the kernel~\eqref{eqn:kernel} for the unanchored RKHS.
For any points  $\bsx_0,\ldots,\bsx_{n-1}\in[0,1]^d$ and any  weights $a_0,\ldots,a_{n-1}\in\bbR$
\begin{align}\label{eq:Roth}
\Bigl\Vert1-\sum_{i=0}^{n-1} a_i K(\bsx_i,\cdot)\Bigr\Vert\geqslant A_d\frac{(\log n)^{(d-1)/2}}{n}
\end{align}
holds for some positive number $A_d$ independent of $n$.
\end{theorem}
\begin{proof}
Because the function $\bsone$ is the Riesz representation of integration over $[0,1]^d$,
$(\bsone,\bsone)=\int_{[0,1]^{d}}\bsone \,\rd\bsy=1,$ and
$$(K(\bsx,\cdot),1)=\prod_{j=1}^d \int_0^1 \frac{4}{3}+\frac{1}{2}\Bigl(x_j^2+y_{j}^2-x_j-y_{j}-|x_j-y_{j}|\Bigr)\,\rd y_j=1.$$
Therefore
$$\Bigl(\bsone-\sum_{i=0}^{n-1} a_i K(\bsx_i,\cdot),\bsone\Bigr)
=(\bsone,\bsone)-\sum_{i=0}^{n-1} a_i (K(\bsx_i,\cdot),\bsone)=1-\sum_{i=0}^{n-1} a_i.$$
On the other hand,
$$\Bigl|\Bigl(\bsone-\sum_{i=0}^{n-1} a_i K(\bsx_i,\cdot),\bsone\Bigr)\Bigr|
\leqslant \Vert\bsone\Vert\times \Bigl\Vert\bsone-\sum_{i=0}^{n-1} a_i K(\bsx_i,\cdot)\Bigr\Vert=\Bigl\Vert\bsone-\sum_{i=0}^{n-1} a_i K(\bsx_i,\cdot)\Bigr\Vert,$$
so we get the first lower bound
\begin{equation}\label{eqn:bound1}
    \Bigl\Vert1-\sum_{i=0}^{n-1} a_i K(\bsx_i,\cdot)\Bigr\Vert\geqslant \Bigl|1-\sum_{i=0}^{n-1} a_i\Bigr|.
\end{equation}
Next we evaluate $\Vert1-\sum_{i=0}^{n-1} a_i K(\bsx_i,\cdot)\Vert^2$ directly.
If we ignore all summands except for $u=\{1,\ldots,d\}$ in \eqref{eqn:innerproduct}, we get
\begin{align}\label{eqn:ainorm}
\Bigl\Vert1-\sum_{i=0}^{n-1} a_i K(\bsx_i,\cdot)\Bigr\Vert^2&\geqslant \int_{[0,1]^d} \bigg(\sum_{i=0}^{n-1} a_i\frac{\partial^{d}K(\bsx_i,\bsy)}{\partial y_1\cdots\partial y_d}\bigg)^2 \,\rd\bsy \nonumber\\
   &= \int_{[0,1]^d} \bigg(\sum_{i=0}^{n-1} a_i\prod_{j=1}^d (y_j-1\{y_j>x_{ij}\})\bigg)^2 \,\rd\bsy
\end{align}
where $x_{ij}$ is the $j$th component of $\bsx_i$.

The left hand side of~\eqref{eqn:ainorm} is a squared norm in the RKHS
while the right hand side is a plain $L^2$ squared norm.
To provide a lower bound, we construct a function $h$ satisfying
\begin{align*}
&\int_{[0,1]^d} h(\bsy)^2 \,\rd\bsy=O((\log n)^{d-1}),\quad\text{and}\\
&\int_{[0,1]^d} h(\bsy)\bigg(\sum_{i=0}^{n-1} a_i\prod_{j=1}^d (y_j-1\{y_j>x_{ij}\})\bigg) \,\rd\bsy=\Omega\Bigl(\frac{(\log n)^{d-1}}{n}\Bigr),
\end{align*}
neither of which involve the RKHS inner product, so the function $h$ does not have to be in the RKHS.

Define $E(\bsk,\bsc)$ to be the $d$-dimensional interval
from~\eqref{eq:elemint} with $b=2$, that is
$$
E(\bsk,\bsc)=\prod_{j=1}^d\Bigl[
\frac{c_j}{2^{k_j}},
\frac{c_j+1}{2^{k_j}}
\Bigr)
$$
where $\bsk = (k_1,\dots,k_d)\in\bbZ^d$
and $\bsc = (c_1,\dots,c_d)\in\bbZ^d$
satisfy $k_j\geqslant0$ and $0\leqslant c_j <2^{k_j}$. Given $\bsk$, we define $|\bsk|=\sum_{j=1}^dk_j$.
For a given vector $\bsk$,
the $2^{|\bsk|}$ elementary intervals $E(\bsk,\bsc)$
partition $[0,1)^d$ into congruent sub-intervals.

For each $E(\bsk,\bsc)$, define $U_{\bsk,\bsc}(\bsy)$ by
$$
U_{\bsk,\bsc}(\bsy)=
\begin{cases}
(-1)^{\sum_{j=1}^d 1\{2^{k_j}y_j-c_j<1/2\}}, & \bsy\in E(\bsk,\bsc)\\
0, &\text{else}.
\end{cases}
$$
If we divide $E(\bsk,\bsc)$ into $2^d$ sub-intervals, the value of $U_{\bsk,\bsc}(\bsy)$ is constant on each sub-interval and it alternates between $1$ and $-1$.

It is straightforward to verify that $\int_{[0,1]^d}U_{\bsk,\bsc}(\bsy)U_{\bsk',\bsc'}(\bsy)\,\rd\bsy=0$ if $\bsk\neq \bsk'$ or $\bsc\neq \bsc'$. It is trivially true if $E(\bsk,\bsc)\cap E(\bsk',\bsc')=\varnothing$.
If instead $E(\bsk,\bsc)\cap E(\bsk',\bsc')\neq \varnothing$, then there must be some $j$ with $k_j<k'_j$ and observe that as a function of $y_j$, $U_{\bsk,\bsc}(\bsy)U_{\bsk',\bsc'}(\bsy)$ equals $1$ on a $1/2^{k_j+1}$-length interval and equals $-1$ on an adjacent $1/2^{k_j+1}$-length interval, so the integration over variable $j$ always returns $0$. Then for any set $P$ of $(\bsk,\bsc)$ pairs,
\begin{equation}\label{eqn:Unorm}
   \int_{[0,1]^d}\bigg(\sum_{(\bsk,\bsc)\in P}U_{\bsk,\bsc}(\bsy)\bigg)^2 \,\rd\bsy=\sum_{(\bsk,\bsc)\in P}\int_{[0,1]^d}U_{\bsk,\bsc}(\bsy)^2 \,\rd\bsy=\sum_{(\bsk,\bsc)\in P} 2^{-|\bsk|}.
\end{equation}

Now let $\mathcal{P}=\{\bsx_0,\ldots,\bsx_{n-1}\}$ and choose $m$ so that $2n\leqslant 2^m<4n$. For any $\bsk$ with $|\bsk|=m$, define the set $P_\bsk$ to be
$$P_{\bsk} =\{\bsc\mid E(\bsk,\bsc)\cap \mathcal{P}=\varnothing\}.$$
For each $\bsc\in P_\bsk$,
$$\int_{[0,1]^d}U_{\bsk,\bsc}(\bsy)\prod_{j=1}^d (y_j-1\{y_j>x_{ij}\}) \,\rd\bsy=\frac{1}{4^{m+d}}.$$
Because there are $2^m$ intervals associated with $\bsk$, the cardinality of $P_{\bsk}$ is at least $2^m-n\geqslant n$. Hence
\begin{equation}\label{eqn:Uprod}
    \int_{[0,1]^d} \bigg( \sum_{\bsc\in P_{\bsk}}  U_{\bsk,\bsc}(\bsy)\bigg)\prod_{j=1}^d (y_j-1\{y_j>x_{ij}\})\,\rd\bsy\geqslant \frac{n}{4^{m+d}}.
\end{equation}
Now we define
$$h(\bsy)=\sum_{\bsk:|\bsk|=m} \sum_{\bsc\in P_{\bsk}} U_{\bsk,\bsc}(\bsy).$$
The number of $\bsk$ with $|\bsk|=m$ is the number of ways to partition $m$ into $d$ nonnegative ordered integers,
which equals ${m+d-1\choose d-1}$. Equation~\eqref{eqn:Unorm} and $2n\leqslant 2^m<4n$ imply that
$$\int_{[0,1]^d} h(\bsy)^2 \,\rd\bsy=\sum_{\bsk:|\bsk|=m} \sum_{\bsc\in P_{\bsk}} 2^{-m}\leqslant \sum_{\bsk:|\bsk|=m}1={m+d-1\choose d-1}\leqslant C_d (\log n)^{d-1}$$
for some positive number $C_d$ independent of $n$. On the other hand, equation~\eqref{eqn:Uprod} implies that
$$    \int_{[0,1]^d} h(\bsy)\prod_{j=1}^d (y_j-1\{y_j>x_{ij}\})\,\rd\bsy\geqslant {m+d-1\choose d-1}\frac{n}{4^{m+d}}\geqslant  \frac{c_d(\log n)^{d-1}}{n}$$
for another positive number $c_d$ independent of $n$.

By the Cauchy-Schwarz inequality and equation~\eqref{eqn:ainorm}
\begin{align*}
   &\int_{[0,1]^d} h(\bsy)\bigg(\sum_{i=0}^{n-1} a_i\prod_{j=1}^d (y_j-1\{y_j>x_{ij}\})\bigg) \,\rd\bsy\\
   &\leqslant \bigg(\int_{[0,1]^d} h(\bsy)^2 \,\rd\bsy\bigg)^{\frac{1}{2}}\bigg(\int_{[0,1]^d} \bigg(\sum_{i=0}^{n-1} a_i\prod_{j=1}^d (y_j-1\{y_j>x_{ij}\})\bigg)^2 \,\rd\bsy\bigg)^{\frac{1}{2}}\\
   &\leqslant \bigg(\int_{[0,1]^d} h(\bsy)^2 \,\rd\bsy\bigg)^{\frac{1}{2}}\Bigl\Vert\bsone-\sum_{i=0}^{n-1} a_i K(\bsx_i,\cdot)\Bigr\Vert
\end{align*}
which provides the lower bound
\begin{align*}
    \Bigl\Vert1-\sum_{i=0}^{n-1} a_i K(\bsx_i,\cdot)\Bigr\Vert
&\geqslant (C_d (\log n)^{d-1})^{-\frac{1}{2}} \frac{c_d(\log n)^{d-1}}{n} \sum_{i=0}^{n-1} a_i=\bigg(\frac{c_d}{C_d^{1/2}}\sum_{i=0}^{n-1} a_i\bigg)\frac{(\log n)^{(d-1)/2}}{n} .
\end{align*}

Combining the above lower bound with equation~\eqref{eqn:bound1} we get
$$\Bigl\Vert\bsone-\sum_{i=0}^{n-1} a_i K(\bsx_i,\cdot)\Bigr\Vert
\geqslant \max\Biggl(\Bigl|1-\sum_{i=0}^{n-1} a_i\Bigr|,\bigg(\frac{c_d}{C_d^{1/2}}\sum_{i=0}^{n-1} a_i\bigg)\frac{(\log n)^{(d-1)/2}}{n} \Biggr).$$
For $\lambda>0$,
$$
\min_{a\in\bbR}\max( |1-a|,\lambda a) = \frac{\lambda}{\lambda+1}
$$
so that
\begin{align*}
    \Bigl\Vert1-\sum_{i=0}^{n-1} a_i K(\bsx_i,\cdot)\Bigr\Vert
&\geqslant
\frac{{c_d(\log n)^{(d-1)/2}}/{(nC_d^{1/2})}}{1+{c_d(\log n)^{(d-1)/2}}/{(nC_d^{1/2})}}
\end{align*}
and we let $A_d =(c_d/C_d^{1/2})/(1+c_dM_d/C_d^{1/2})$
for $M_d = \sup_{n\in\bbN}\log(n)^{(d-1)/2}/n$.
\end{proof}
\begin{corollary}
For any sequence of points $(\bsx_i)_{i\geqslant0}$ in $[0,1]^d$, there exists a function $f$ in the RKHS with kernel defined by \eqref{eqn:kernel} such that
$$\limsup_{n\to \infty}\frac{\big|\int_{[0,1]^{d}}f(\bsx) \,\rd\bsx-\frac{1}{n}\sum_{i=0}^{n-1} f(\bsx_i)\big|}{(n\log\log n)^{-1}(\log n)^{(d-1)/2}}=+\infty$$
\end{corollary}
\begin{proof}
Apply Theorem~\ref{thm:asymptoticrate} with the lower bound on $r_n$ from Theorem~\ref{thm:Roth}.
\end{proof}


\section{Discrepancy and the case of $d=1$}\label{sec:d=1}

Let $x_0,x_1,\dots,x_{n-1}\in[0,1]$.
The local discrepancy of these points at $\alpha\in[0,1]$ is
$$
\delta_n(\alpha) = \frac1n\sum_{i=0}^{n-1}1\{x_i<\alpha\}-\alpha
$$
and the star discrepancy is $D_n^*=\sup_{0\leqslant\alpha\leqslant1}|\delta_n(\alpha)|$.
No infinite sequence $x_i$ can have $D_n^*=o(\log(n)/n)$.
Using results from discrepancy theory we see below that there are specific
values of $\alpha$ for which $D_n^*=\Omega(\log(n)/n)$.
For those values $1\{x<\alpha\}-\alpha$ is a persistent
fooling function.  We show below that $f(x)=x$ is also a persistent
fooling function for the van der Corput sequence.

The set of $\alpha\in[0,1]$
with $|\delta_n(\alpha)|=o(\log(n)/n)$ has Hausdorff dimension 0
for any sequence $(x_i)_{i\geqslant0}\subset[0,1]$. See \cite{hala:1981}.
So $r=1$ is not just available for $d=1$ it is the usual
rate for functions of the form $f(x) = 1\{x<\alpha\}$.

For $x_i$ taken from the van der Corput sequence,
and $\alpha = \sum_{k=1}^\infty a_k/2^k$ for bits $a_k\in\{0,1\}$,
Drmota, Larcher and Pillichshammer \cite{drmo:larc:pill:2005} note
that $n|\delta_n(\alpha)|$
is bounded as $n\to\infty$, if and only if $\alpha$ has a representation
with only finitely many nonzero $a_k$.  Further, letting
$$h_\alpha(m) = \#\{ k<m\mid a_k\ne a_{k+1}\}$$
their Corollary 1 in our notation has
\begin{align}\label{eq:theircor1}
\lim_{m\to\infty}\frac1{2^m}\#\bigl\{
1\leqslant n \leqslant 2^m \mid n\delta_n > (1-\epsilon) h_\alpha(m)
\bigr\}=1
\end{align}
for any $\epsilon >0$.
The base $2$ representation of $2/3$ is $0.10101\cdots$
and so $h_{2/3}(m)=m$.
It follows that $f(x) = 1\{x<2/3\}$ has
$|\hat\mu_n-\mu|>c\log(n)/n$ infinitely often for some $c>0$.
Even more, the fraction of such $n$ among the first $N=2^m$
sample sizes becomes ever closer to $1$ as $m\to\infty$.

If we average the local discrepancy over $\alpha$ we get
$$
\int_0^1\delta_n(\alpha)\,\rd\alpha = \frac1n\sum_{i=0}^nx_i-\frac12
$$
which is the integration error for the function $f(x)=x$
that we study next.
In our study of $f(x)=x$, we use sample sizes $n$
with base $2$ expansion $10101\cdots101$.
That is, for some $L\geqslant1$
$$n=n_L = \sum_{\ell=0}^L4^{\ell}.$$
The first few values of $n_L$ are
$1$, $5$, $21$, $85$, $341$, $1365$, and $5461$.

\begin{proposition}\label{prop:sumx}
For integers $0\leqslant i<n$, let $x_i$ be the $i$'th van der Corput
point in base $2$.
Then
\begin{align}\label{eq:sumx}
\sum_{i=0}^{n-1}x_i =
\sum_{k=1}^\infty
\lfloor 2^{-k-1}n\rfloor+(2^{-k}n-2\lfloor 2^{-k-1}n\rfloor-1)_+.
\end{align}
\end{proposition}
\begin{proof}
As $i$ increases from $0$, the $k$'th digit of $x_i$
comes in alternating blocks of $2^k$ zeros and $2^k$ ones,
starting with zeros.  For $0\leqslant i<r2^{k+1}$ with
an integer $r\geqslant0$, the $k$'th digits sum to $r2^k$
because there are $r$ complete blocks of $2^k$ ones.
The number of ones among $i_k$ for $0\leqslant i<n$ is then
$$
2^k\lfloor 2^{-k-1}n\rfloor+(n-2^{k+1}\lfloor 2^{-k-1}n\rfloor-2^k)_+
$$
where $z_+=\max(z,0)$.  The first term counts ones
from $\lfloor 2^{-k-1}n\rfloor$ complete blocks of $2^{k+1}$
indices.  That leaves an incomplete block of
$n-2^{k+1}\lfloor 2^{-k-1}n\rfloor$ indices $i$
of which the first $2^k$, should there be that many,
must be $0$s.  Any indices past the first $2^k$ are
ones providing the second term above.
To complete the proof, the $k$'th digits have a coefficient
of $2^{-k}$ in $x_i$ and summing over digits yields~\eqref{eq:sumx}.
\end{proof}

The sum over $k$ in~\eqref{eq:sumx} only needs to go
as far as the number of nonzero  binary digits
in $n-1$.
For larger $k$, both parts of the $k$'th term are zero.

\begin{proposition}
For $L\geqslant0$,  let $n_L = \sum_{\ell=0}^L4^\ell$.
For integers $i\geqslant0$ let $x_i$ be the van der Corput points in base $2$
and for $n\geqslant1$ let $\hat\mu_{n}=(1/n)\sum_{i=0}^{n-1}x_i$.
Then
$$
\limsup_{L\to\infty} \frac{n_L}{\log(n_L)} |\hat\mu_{n_L}-1/2|
>c
$$
if $0<c<1/(8\log(2))$.
\end{proposition}
\begin{proof}
We need $K(n_L)=2L-1$ base $2$ digits to represent $n_L$.
We can write $n_L = \sum_{\ell=0}^L2^{2\ell}$.
Then for $1\leqslant k\leqslant 2L-1$,
\begin{align*}
\lfloor 2^{-k-1}n_L\rfloor
&
=\Bigl\lfloor2^{-k-1}\sum_{\ell=0}^L2^{2\ell}\Bigr\rfloor
=2^{-k-1}n_L-\sum_{\ell=0}^{\lfloor (k-1)/2\rfloor}2^{2\ell-k-1}.
\end{align*}
Next let
\begin{align*}
\theta_k\equiv
\sum_{\ell=0}^{\lfloor (k-1)/2\rfloor}2^{2\ell-k-1}
=2^{-k-1}\sum_{\ell=0}^{\lfloor (k-1)/2\rfloor}2^{2\ell}
= 2^{-k-1}\frac{4^{\lfloor (k+1)/2\rfloor}-1}3
\end{align*}
and then
\begin{align*}
(2^{-k}n_L-2\lfloor 2^{-k-1}n_L\rfloor-1)_+
&=
(2^{-k}n_L-2^{-k}n_L+2\theta_k-1)_+
=(2\theta_k-1)_+
\end{align*}
Because $\theta_k<1/3$ we have
$(2\theta_k-1)_+=0$.
Using Proposition~\ref{prop:sumx},
\begin{align*}
\frac1{n_L}\sum_{i=0}^{n_L-1}x_i
&=\frac1{n_L}\sum_{k=1}^{K(n_L)}\lfloor 2^{-k-1}n_L\rfloor
=\frac1{n_L}\sum_{k=1}^{K(n_L)}
2^{-k-1}n_L-\theta_k\\
&=
\frac12-2^{-K(n_L)-1}-\frac1{n_L}\sum_{k=1}^{K(n_L)}\theta_k.
\end{align*}
Now because $\theta_k\geqslant\theta_2=1/8$,
$|\hat\mu_{n_L}-1/2| \geqslant %
K(n_L)/(8n_L)$
and so
\begin{align*}
\frac{n_L}{\log(n_L)}|\hat\mu_{n_L}-1/2|
&\geqslant \frac{K(n_L)}{8\log(n_L)}
> \frac18\frac{2L-1}{(L+1)\log(4)}\to\frac1{4\log(4)},
\end{align*}
completing the proof.  
\end{proof}


We can see why the sample sizes $n_L$ give unusually innaccurate estimates
of the integral of $x$ in the van der Corput sequence.
Those values of $n$ consistently miss getting into the `ones block'
for digit $k$.

Figure~\ref{fig:vdcempiricals} shows some empirical
behavior of the scaled errors for the two integrands
we considered in this section. 
The scaled error there is essentially the number of observations
by which the count of points in $[0,2/3)$ differs from
$2n/3$.  It is $n|\delta_n(2/3)|$ which is the customary
scaling in the discrepancy literature.

\begin{figure}[t]
\centering
\includegraphics[width=.9\hsize]{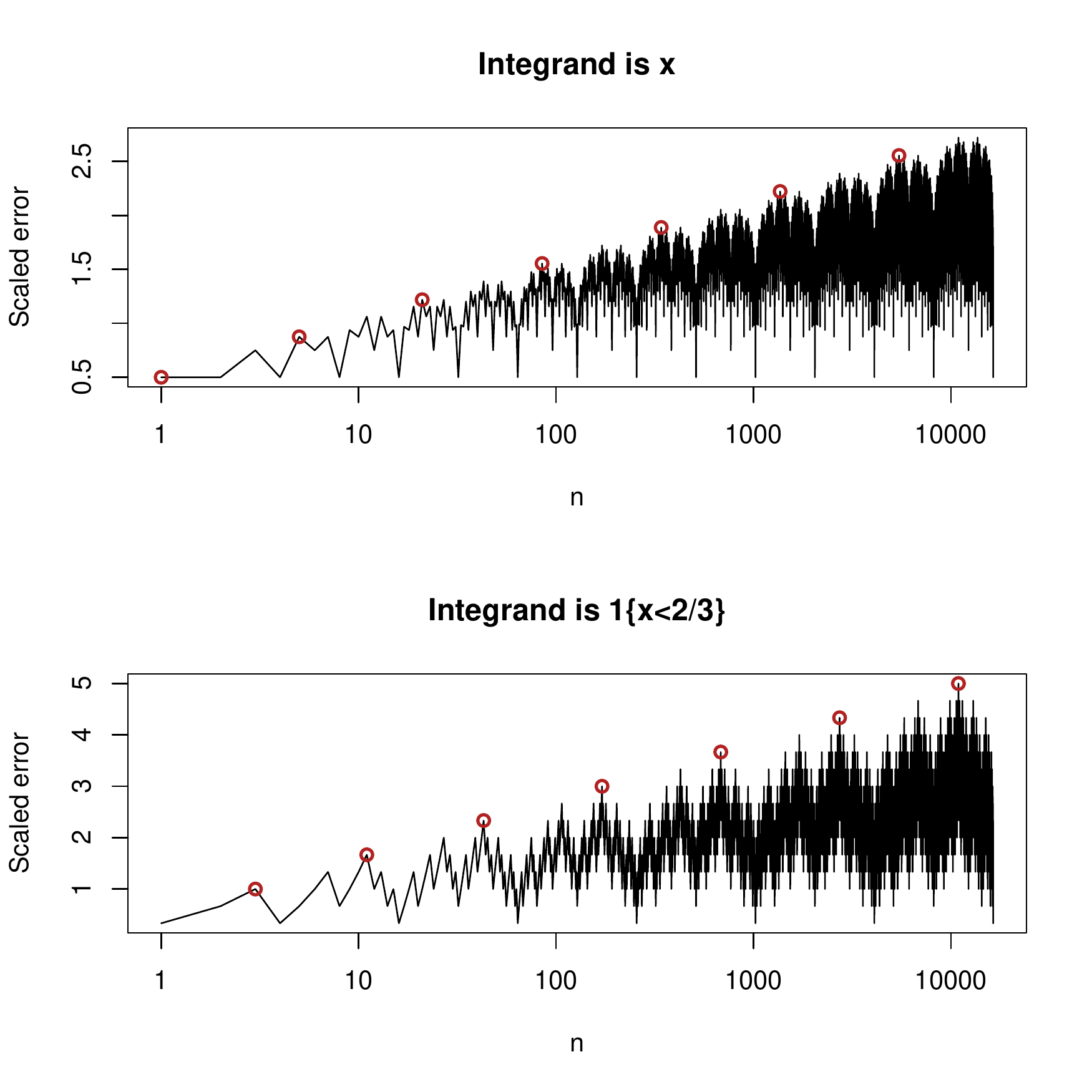}
\caption{\label{fig:vdcempiricals}
Both panels show the scaled error $n|\hat\mu_n-\mu|$ versus $n$
for the first $2^{14}$ points of the van der Corput sequence.
The top panel has the integrand $f(x)=x$ and the values
for $n=n_L$ are indicated with open circles.  The bottom panel has the integrand $f(x) = 1\{x<2/3\}$
and the values $n=\tilde n_L$ are indicated.
}
\end{figure}

The integrands $x$ and $1\{x<2/3\}$ both have
total variation one.  It would be interesting to know
what integrand of total variation one has the largest
value for $\limsup_{n\to\infty} n|\hat\mu_n-\mu|/\log(n)$
in the van der Corput sequence.

For integration, it is advisable to use $n=b^m$ for $m\geqslant0$
in the van der Corput sequence. Then the Koksma-Hlawka inequality
gives us $|\hat\mu_{b^m}-\mu|=O(1/b^m)$ as $m\to\infty$
for any $f$ of bounded variation on $[0,1]$ because
the van der Corput sequence has $D^*_{b^m} = b^{-m}=1/n$.
Any $(t,d)$-sequence for $d=1$ has $D^*_{b^m}=O(1/n)$.
For $d=1$, bounded variation in the sense of Hardy
and Krause reduces to the familiar one dimensional notion
of bounded variation.
As a result a log power $r>0$ can apply to
the limit as $n\to\infty$ through positive
integers but not as $n=b^m$ for $m\to\infty$.

\section{Empirical investigations for $d=2$}
\label{sec:d=2}

This section reports on an empirical search for
an integrand with errors that are $\Omega( \log(n)^r/n)$
for some $r>1$ when using a sequence of points
with $D_n^*=O(\log(n)^2/n)$.  We know from Section~\ref{sec:proofofbound}
that this is attainable for $1<r<3/2$
and not ruled out for $3/2\leqslant r<2$.
We search using some knowledge of the
weaknesses of the van der Corput points for the functions
from Section~\ref{sec:d=1}.

The search must take place with $d\geqslant2$ and so $d=2$
is the natural first place to look for a problematic integrand.
It is clear that the integrand should not be additive because
then integrating it involves a sum of two one dimensional
integration problems where we know that $r>1$ is not needed.
We look at two generalizations of the van der Corput
sequence.  The first is the Halton sequence for $d=2$.
In that sequence, $\bsx_{i1}$ is the van der Corput sequence
in base $2$ and $\bsx_{i2}$ is the van der Corput sequence
in base $3$.
The second sequence is the Sobol' sequence in base $2$.
It has $t=0$ and like the Halton points, the first
component $\bsx_{i1}$ is the van der Corput sequence
in base $2$.  For $n=2^m$, the second component
$\bsx_{i2}$ of the Sobol' sequence is a permutation
of the first $n$ elements of the van der Corput sequence.


We look first at the scaled error $n|\hat\mu_n-\mu|$
for $f(\bsx_i) = (x_{i1}-1/2) (x_{i2}-1/2)$.
This integrand has integral zero and two dimensional Vitali variation of one.
It integrates to zero over each component of $\bsx$ when the other
component is fixed at any value.  It has no nonzero ANOVA component
of dimension smaller than $2$.  Each of the factors is a persistent
fooling function for the van der Corput points.

For the Halton sequence, the scaled error equals $1/4$ when $n=1$
and it remains below $1/4$ for $2\leqslant n\leqslant 2^{20}$.  It repeatedly approaches $1/4$
from below.
Figure~\ref{fig:d2prodempirical} shows the first $2^{19}$ scaled errors
for this product integrand for both Halton and Sobol' points.
For both constructions we see no evidence that the error is
$\Omega(\log(n)^r/n)$ for $r>1$.  In fact, what is quite surprising
there is the apparent $O(1/n)$ error rate, which is then
{\sl better} than what the van der Corput sequence attains for $f(x)=x$.

By plotting only $2^{19}$ points
instead of all $2^{20}$ it becomes possible to see something of
the structure within the triangular peaks for the Sobol' points.
The linear scale for $n$ (versus a logarithmic one) shows
a clear repeating pattern consistent with an $O(1/n)$ error rate.
Whether or not the errors are $O(1/n)$, this integrand is
not a promising one to investigate further in the
search for an integrand needing $r>1$.

\begin{figure}[t]\centering
\includegraphics[width=.9\hsize]{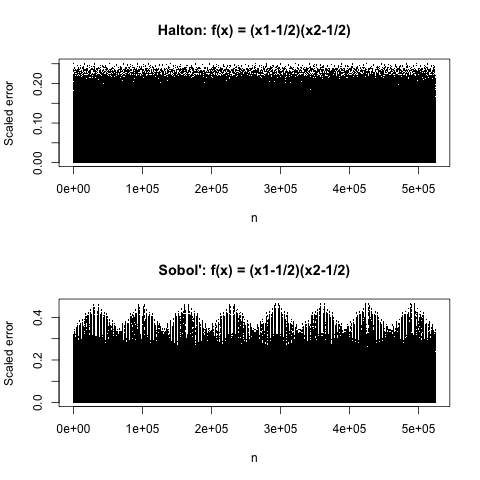}
\caption{\label{fig:d2prodempirical}
The panels show scaled errors $n|\hat\mu_n-\mu|$ versus $n$
for the integrand $f(\bsx)=(x_1-1/2) (x_2-1/2)$ on $[0,1]^2$.
The top panel is for Halton points and the bottom one is
for Sobol' points.
}
\end{figure}

Another challenging integrand for van der Corput
points was $1\{x<2/3\}$.
For the Sobol' points we then take
$f(\bsx)=\prod_{j=1}^2( 1\{x_j<2/3\}-2/3)$
for $d=2$.  Once again we have removed
the additive contribution to the integrand
that we know is integrated at the $\log(n)/n$
rate making it easier to discern the effect
of the bivariate component.
For Halton points, we do not use $2/3$
as that has a terminating expansion in base $3$
that defines the second component of the $\bsx_i$.
We use
$f(\bsx)=(1\{x_1<2/3\}-2/3) (1\{x_2<3/5\}-3/5)$
for Halton points because $3/5$ does not have a terminating
expansion in base $3$ that could make the problem
artificially easy.  Both of these functions have
two dimensional Vitali variation equal to one, just
like $\prod_{j=1}^2(x_j-1/2)$.
Figure~\ref{fig:d2rectempirical} shows the scaled errors
$n|\hat\mu-\mu|/\log(n)$. The logarithmic scaling
in the denominator is there so that
a value of $r>1$ would give an infinite series of new records.
We don't see many such new records for $n\leqslant 2^{22}$
for either of the two test cases. These integrands were designed
to exploit weaknesses in the Sobol' and Halton sequences
and they did not produce the appearance of errors growing
faster than $\log(n)/n$.  It remains possible that errors
grow like $\log(n)^r/n$ for $r>1$ but with the records being set very sparsely.
The situation is quite different from Figure~\ref{fig:vdcempiricals}
in the one dimensional case where we see a steady sequence of
record values with a fractal pattern in the empirical errors.

\begin{figure}[t]\centering
\includegraphics[width=.9\hsize]{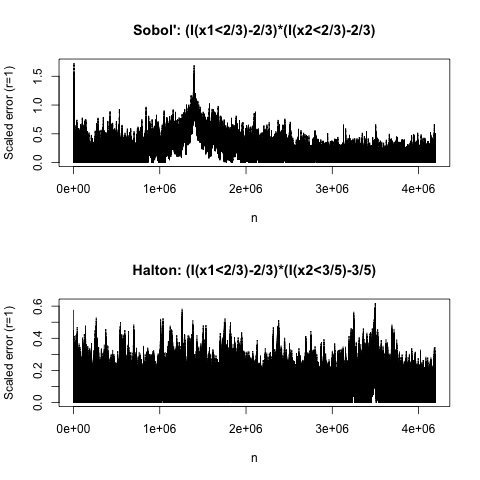}
\caption{\label{fig:d2rectempirical}
The panels show scaled errors $n|\hat\mu_n-\mu|/\log(n)$ versus $n>1$.
The top panel has $f(\bsx)=(1\{x_1<2/3\}-2/3)\times(1\{x_2<2/3\}-2/3)$
for Sobol' points.
The bottom panel has
$f(\bsx)=(1\{x_1<2/3\}-2/3)\times(1\{x_2<3/5\}-3/5)$
for Halton points.
}
\end{figure}

There are some other integrands that could be difficult for
QMC.  A badly approximable irrational number $\theta$
is one where the distance between $n\theta$
and $\Z$ is above $c/n$ for some $c>0$ and all integers $n\geqslant1$.
For instance, $\theta=\sqrt{2}-1$ is badly approximable.
An integrand like $\prod_{j=1}^d (1_{x_j<\theta}-\theta)$
could pose a challenge to QMC methods, but some exploration
with Halton and Sobol' points was inconclusive: there was
no empirical indication that $r>1$ would be required.
Another potentially difficult integrand is $\prod_{j=1}^d (x_j^\theta-1/(1+\theta))$
for $0<\theta<1$. Partial derivatives of this integrand with respect to a subset
of components $x_j$ are unbounded.  That would make them fail the sufficient
conditions for scrambled nets to attain $O(n^{-3/2+\epsilon})$ RMSE
in \cite{smoovar,localanti,yue:mao:1999}.
These integrands did not show a need for $r>1$.

Next, we consider the function $f(\bsx) = 1\{x_1+x_2<1\}$.
This integrand has infinite Vitali variation
over $[0,1]^2$. Therefore it also has infinite Hardy-Krause
variation  and so the Koksma-Hlawka bound degenerates to $+\infty$.
Because $f$ is Riemann integrable we know that $\hat\mu_n\to\mu_n$
if $D_n^*\to0$.
Finding that $r>1$ for this $f$ would not provide
an example of a BVHK function needing that $r>1$.
It remains interesting to see what happens because the case is
not covered by QMC theory without randomization.

Figure~\ref{fig:d2notbvempirical} shows scaled errors
for this integrand, using an exponent of $r=1$  for $\log(n)$.
There is a striking difference between the results for Sobol'
points versus Halton points.  We see a few approximate doublings of the
scaled error for Sobol' points even when using $r=1$.
This does not prove that
we need $r>1$ here; perhaps we saw the last doubling
or perhaps similar jumps for larger $n$ arise at extremely sparse intervals.
It does serve to raise some additional
questions.  For instance, why are Halton points so much more
effective on this integrand, and to which other integrands
of unbounded variation might that apply?
For smooth integrands, Sobol' points commonly perform
much better when $n$ is a power of two than otherwise.
Here, those sample sizes did not bring much better outcomes.

\begin{figure}[t]\centering
\includegraphics[width=.9\hsize]{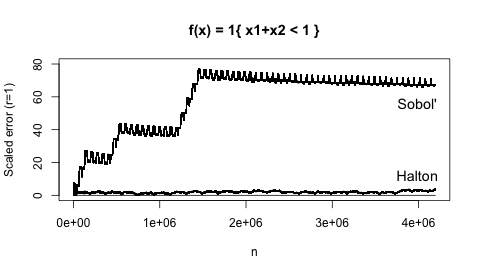}
\caption{\label{fig:d2notbvempirical}
This shows
scaled errors $n|\hat\mu_n-\mu|/\log(n)$ versus $n>1$.
The integrand is $1\{x_1+x_2<1\}$ which has infinite
Vitali variation on $[0,1]^2$. The upper curve is for Sobol' points.
The lower curve is for Halton points. There is a reference
line at scaled error of zero.
}
\end{figure}

\section{Very large $m$ for Sobol' nets}\label{sec:bigm}

If the need for $r>1$ is only evident for very large $n$
then we might fail to detect it by computing $\hat\mu_n$.
If we restrict to sample sizes $n=2^m$  for $m\geqslant0$ then
we can use properties of the generating matrices of Sobol'
sequences to compute the scaled error $n|\hat\mu-\mu|$
for very large $n$ when $f$ is the indicator function
of a set $[\bszero,\bsa)$.
The first $n=2^m$ Sobol' points form a $(t,m,s)$-net in
base $2$ for which the Koksma-Hlawka bound gives
a rate of $O(n^{-1}\log(n)^{d-1})$.  As a result we must
look to $d=3$ or larger for a problem that needs $r>1$
for these values of $n$.
We choose $\bsa = (2/3,2/3,\dots,2/3)$ of length $d$
as $2/3$ is difficult to approximate in base $2$.

We can partition $[0,2/3)^d$ into a countable number
of elementary intervals in base $2$.
The number of points of the digital net
$\bsx_0,\dots,\bsx_{b^m-1}$
that are in $E(\bsk,\bsc)$ equals the number of solutions
$\vec{i}\in \{0,1\}^m$ to a linear equation
$C\vec{i}=\vec{a} \mathrm{\,mod\,} 2$
for a matrix $C\in\{0,1\}^{|\bsk|\times m}$
that takes the first $k_j$ rows from the $j$'th
generating matrix of the digital net, for $j=1,\dots,d$
and some $\vec{a}\in\{0,1\}^m$.
See \cite{pan:owen:2021:tr} for a description and
a discussion of how to find the number of such points.
For our Sobol' points we use  the direction
numbers from Joe and Kuo \cite{joe:kuo:2008}.

To make the computation finite,
we replace $[0,2/3)^d$ by $[0,a_m)^d$
where $a_m=\lfloor 2^m (2/3)\rfloor/2^m$.
For $a\in(0,1)$, let $\mu(a) = a^d$
and $\hat\mu(a) =(1/n)\sum_{i=0}^{n-1}1_{\bsx_i\in[0,a)^d}$.
Then
\begin{align*}
0&\leqslant \mu\Bigl(\frac23\Bigr)-\mu(a_m)
= \Bigl(\frac23-a_m\Bigr)\sum_{j=0}^{d-1}
\Bigl(\frac23\Bigr)^ja_m^{d-j-1}
\leqslant \Bigl(\frac23\Bigr)^{d-1}\frac{d}n.
\end{align*}
The number of the first $n$ Sobol'
points that belong to $[0,2/3)^d\setminus[0,a_m)^d$
is at most $d$. Therefore
\begin{align*}
0&\leqslant
\hat\mu(2/3)-\hat\mu(a_m)\leqslant d/n.
\end{align*}
Now
\begin{align*}
&\quad n(\hat\mu(2/3)-\mu(2/3))
-n(\hat\mu(a_m)-\mu(a_m))\\
&=
n(\hat\mu(2/3)-\hat\mu(a_m)
-n(\mu(2/3)-\mu(a_m))\\
&\in [-(2/3)^{d-1}d,d]\subseteq [-4/3,d]
\end{align*}
so the absolute error in using $n((\hat\mu(a_m)-\mu(a_m))$
instead of $n((\hat\mu(2/3)-\mu(2/3))$
is at most $d$.

Figure~\ref{fig:errorvsmd} shows the results for $d=2,3,4$.
For $d=2$ we see strong linear trend
lines in the scaled error consistent with
needing $r=1$.  For $d=3,4$ the trend is not
obviously linear but it does not have a compelling
and repeating $r>1$ pattern even at $n=2^{50}$ for $d=4$
or $n=2^{100}$ for $d=3$.

\begin{figure}[t]
\centering
\begin{minipage}{0.34\hsize} 
\includegraphics[width=\textwidth]{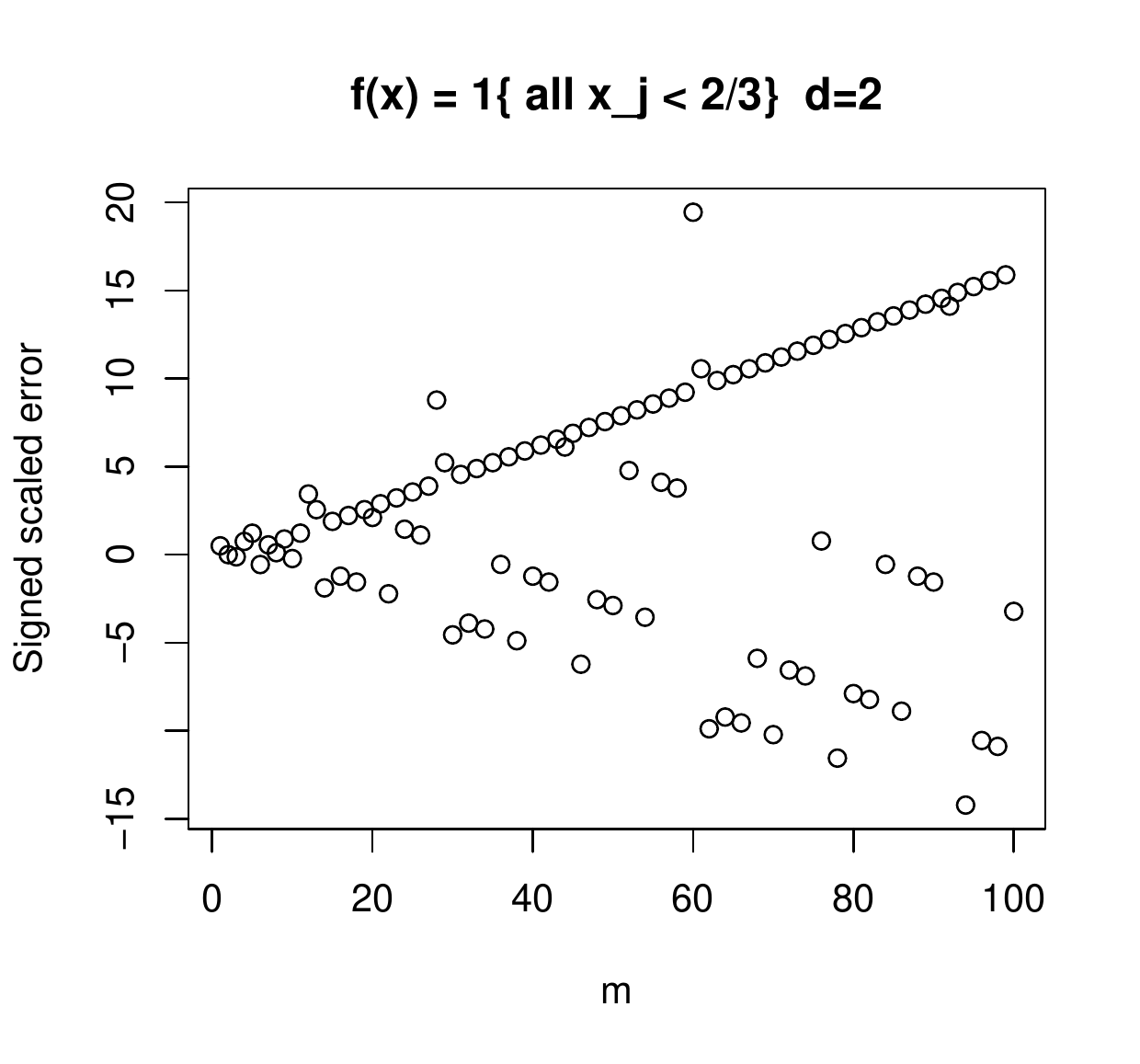}
\end{minipage}
\hspace*{-.35cm} 
\begin{minipage}{0.34\hsize}
\includegraphics[width=\textwidth]{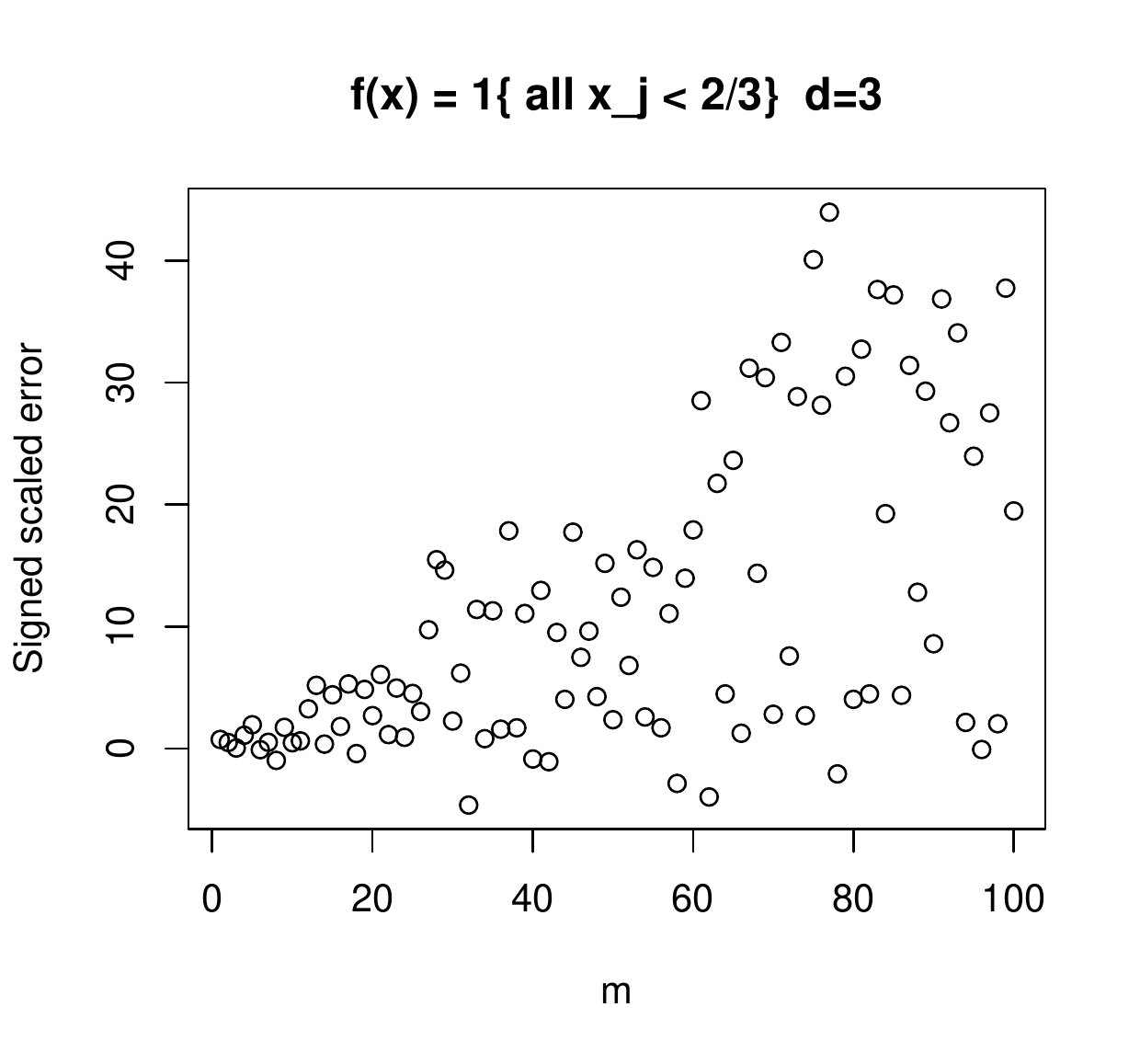}
\end{minipage}
\hspace*{-.35cm} 
\begin{minipage}{0.34\hsize}
\includegraphics[width=\textwidth]{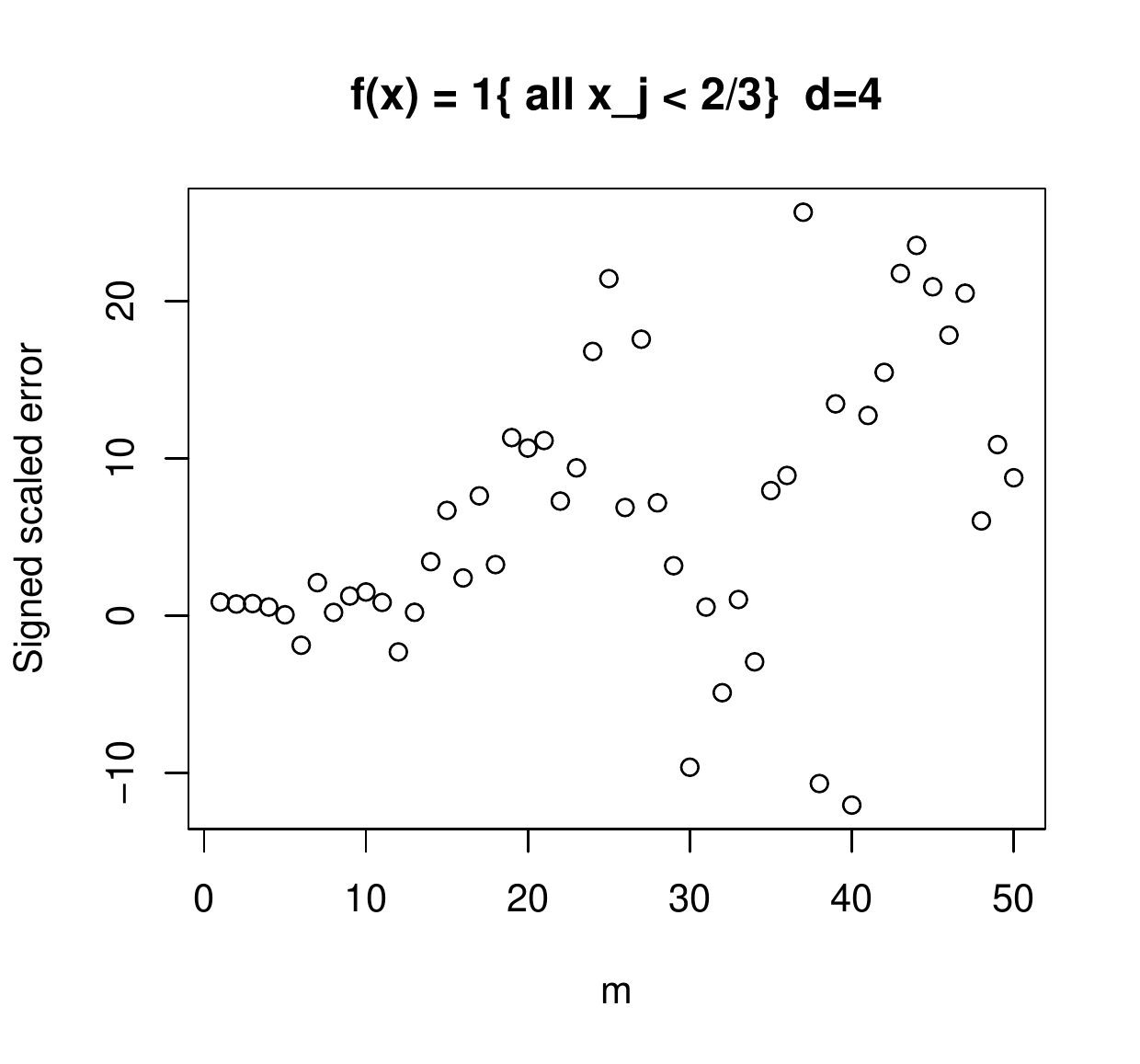}
\end{minipage}
\caption{\label{fig:errorvsmd}
The panels show the signed scaled error
$n(\hat\mu-\mu)$ for a Sobol' sequence
when $f=\prod_{j=1}^d1_{x_j<2/3}$.
The panels have $d=2,3,4$.
The sample sizes go to $2^{100}$ for
$d=2,3$ and to $2^{50}$ for $d=4$.
The computed values differ from the true
ones by at most $d$ for all $m$.
}
\end{figure}

\section{Discussion}\label{sec:discussion}

For any $\epsilon>0$
and any $d\geqslant1$ and any infinite sequence of points $\bsx_i\in[0,1]^d$
and any $c>0$ that there are functions in BVHK$[0,1]^d$ where the QMC error
exceeds $c(\log n)^{r}/n$ infinitely often when $r<(d-1)/2$.
For $d=1$ we can easily construct such functions for $r=1$
using results from discrepancy theory.  We don't know of
any specific examples with $r>1$ and simple multidimensional
generalizations of the functions needing $r=1$ for $d=1$
did not show an apparent need for $r>1$ when $d=2$.
The best candidates we saw for the Sobol' sequence
are $f(\bsx) =\prod_{j=1}^d1\{x_j<2/3\}$ for $d=3$ or $4$
but we have no proof that they require $r>1$.


One surprise is that comparing
Figures~\ref{fig:vdcempiricals} and~\ref{fig:d2prodempirical} we see an error for $f(\bsx)=(x_1-1/2)(x_2-1/2)$ that appears
to be $O(1/n)$ while the one dimensional
function $f(x)=x$ has error $\Omega(\log(n)/n)$ (theoretically
and also empirically over small $n$).
A second surprise is that for a two dimensional function
of unbounded variation we see very different behavior
for Sobol' and Halton points in
Figure~\ref{fig:d2notbvempirical}.  The error for Sobol'
points appears to grow faster than $\log(n)/n$
while that for Halton points is far lower and might
even grow at a different rate.  Neither of these surprises
contradict known theory but it is odd to see the two
dimensional problem apparently easier to solve than
a corresponding one dimensional one and it is also odd
to see Halton points appear to be so much more robust
to unbounded variation than Sobol' points.

So, where could the logs be?  It is possible that they are only
practically important for enormous $n$, or what is almost
the same thing, that they are present with a very tiny implied constant.
It is also possible that the commonly investigated integrands
have error $O(\log(n)/n)$ even for $d\geqslant2$.

\section*{Acknowledgments}
This work was supported by the U.S.\ NSF under
grant IIS-1837931.  Thanks to Fred Hickernell and
Erich Novak for discussions related to this problem.
We are also grateful to Traub, Wasilkowski and Wo\'zniakowski
for ensuring that Trojan's work was not lost.

\bibliographystyle{plain}
\bibliography{qmc}

\end{document}